\documentclass{amsart}
\usepackage{hyperref}
\usepackage{url}
\usepackage{amssymb}

\newtheorem{theorem}{Theorem}[section]
\newtheorem{corollary}{Corollary}[section]

\newtheorem{lemma}{Lemma}[section]

\theoremstyle{definition}

\theoremstyle{remark}

\numberwithin{equation}{section}

\begin{document}

\title[Arithmetic Properties of Partition Triples  With Odd Parts Distinct]
 {Arithmetic Properties of Partition Triples  \\ With Odd Parts Distinct}

\author{LIUQUAN WANG}
\address{Department of Mathematics, National University of Singapore, Singapore, 119076, SINGAPORE}
\address{Department of Mathematics, Zhejiang University, Hangzhou, 310027, CHINA}

\email{mathlqwang@163.com; wangliuquan@u.nus.edu}



\thanks{This paper appears in International Journal of Number Theory, doi:10.1142/S1793042115500773.}
\dedicatory{}

\maketitle

\begin{abstract}
Let $\mathrm{pod}_{-3}(n)$ denote  the number of partition triples of $n$ where the odd parts in each partition are distinct. We find many  arithmetic properties of $\mathrm{pod}_{-3}(n)$ involving the following infinite family of congruences: for any integers $\alpha \ge 1$ and $n\ge 0$,
\[\mathrm{pod}_{-3}\Big({{3}^{2\alpha +2}}n+\frac{23\times {{3}^{2\alpha +1}}+3}{8}\Big)\equiv 0  \pmod{9}.\]
We also establish some arithmetic relations between $\mathrm{pod}(n)$ and $\mathrm{pod}_{-3}(n)$, as well as  some congruences for $\mathrm{pod}_{-3}(n)$ modulo 7 and 11.
\end{abstract}

\section{Introduction}	

 Let $\mathrm{pod}_{-k}(n)$ be defined by the generating function
\[\sum\limits_{n=0}^{\infty }{\mathrm{pod}_{-k}(n){{q}^{n}}}=\frac{(-q;{{q}^{2}})_{\infty }^{k}}{({{q}^{2}};{{q}^{2}})_{\infty }^{k}}=\frac{1}{\psi {{(-q)}^{k}}},\]
where $\psi(q)$ is  one of Ramanujan's theta functions, namely
\[\psi (q)=\sum\limits_{n=0}^{\infty }{{{q}^{n(n+1)/2}}}=\frac{({{q}^{2}};{{q}^{2}})_{\infty }^{2}}{{{(q;q)}_{\infty }}}.\]

The combinatorial interpretation is that $\mathrm{pod}_{-k}(n)$ denotes the number of partition $k$-tuples of $n$ where in each partition the odd parts are distinct. For $k=1$, $\mathrm{pod}_{-1}(n)$ is often denoted as $\mathrm{pod}(n)$.

In 2010, Hirschhorn and Sellers \cite{hisc} investigated the arithmetic properties of $\mathrm{pod}(n)$. They obtained some interesting congruences. For example, for integers $\alpha \ge 0$ and $n\ge 0$, they showed that
\[\mathrm{pod}\Big({{3}^{2\alpha +3}}n+\frac{23\times {{3}^{2\alpha +2}}+1}{8}\Big)\equiv 0 \pmod{3}.\]
Some internal congruences were also found, such as
\[ \mathrm{pod}(81n+17)\equiv 5 \, \mathrm{pod}(9n+2) \pmod{27}. \]

In 2011, using the tool of modular forms, Radu and Sellers \cite{radu} obtained other deep congruences for $\mathrm{pod}(n)$ modulo $5$ and $7$, such as
\[\mathrm{pod}(135n+8)\equiv \mathrm{pod}(135n+107)\equiv \mathrm{pod}(135n+116)\equiv 0 \pmod{5}.\]
\[\mathrm{pod}(567n+260)\equiv \mathrm{pod}(567n+449)\equiv 0 \pmod{7}.\]
Chen and Xia \cite{wyc} followed by studying the arithmetic properties of  $\mathrm{pod}_{-2}(n)$. They found two infinite families of  congruences modulo 3 and 5, respectively. For more properties of $\mathrm{pod}(n)$, see \cite{lovejoy,Wang3}.

Inspired by their work, we study the  arithmetic properties of $\mathrm{pod}_{-3}(n)$, the number of partition triples of $n$ where the odd parts in each partition are distinct.  In Section 2, we will present an infinite family of Ramanujan-type congruences: for any integers $\alpha \ge 1$ and $n\ge 0$,
\[\mathrm{pod}_{-3}\Big({{3}^{2\alpha +2}}n+\frac{23\times {{3}^{2\alpha +1}}+3}{8}\Big)\equiv 0  \pmod{9}.\]
In Section 3, we establish the following arithmetic relations between $\mathrm{pod}(n)$ and $\mathrm{pod}_{-3}(n)$:
\begin{align} \nonumber
  & \mathrm{pod}_{-3}(3n+2)\equiv 6 \, \mathrm{pod}(9n+5)  \pmod{9}, \\ \nonumber
 & \mathrm{pod}_{-3}(3n+1)\equiv 3 \, \mathrm{pod}(9n+2) \pmod{18}.
\end{align}
In Section 4, we will prove
\[\mathrm{pod}_{-3}(7n+3)\equiv {{(-1)}^{n+1}}{{r}_{3}}(8n+3) \pmod{7},\]
\[\mathrm{pod}_{-3}(11n+10)\equiv {{4(-1)}^{n}}{{r}_{7}}(8n+7)  \pmod{11},\]
where $r_{k}(n)$ denotes the number of representations of $n$ as sum of $k$ squares. Many congruences will also be deduced from these two relations.

\section{An Infinite Family of Congruences Modulo 9}

Before our discussion,  we observe that (see \cite{wyc,hisc})
\[\psi (q)=1+q+{{q}^{3}}+{{q}^{6}}+{{q}^{10}}+{{q}^{15}}+\cdots =A({{q}^{3}})+q\psi ({{q}^{9}}).\]
where
\[A(q) =\frac{{{({{q}^{2}};{{q}^{2}})}_{\infty }}({{q}^{3}};{{q}^{3}})_{\infty }^{2}}{{{(q;q)}_{\infty }}{{({{q}^{6}};{{q}^{6}})}_{\infty }}}.\]

The following lemma plays a  key role  in  both \cite{wyc} and \cite{hisc}, which will also be used frequently in our paper.
\begin{lemma}\label{identity1}(Cf.\ \cite[Lemma 2.1]{hisc}.)
We have
\[A{{({{q}^{3}})}^{3}}+{{q}^{3}}\psi {{({{q}^{9}})}^{3}}=\frac{\psi {{({{q}^{3}})}^{4}}}{\psi ({{q}^{9}})}\]
and
\[
\frac{1}{\psi (q)}=\frac{\psi ({{q}^{9}})}{\psi {{({{q}^{3}})}^{4}}}\left(A{{({{q}^{3}})}^{2}}-qA({{q}^{3}})\psi ({{q}^{9}})+{{q}^{2}}\psi {{({{q}^{9}})}^{2}}\right).
\]
\end{lemma}
For convenience, throughout the paper, we denote $s=A({{q}^{3}})$ and $t=\psi ({{q}^{9}})$. Then Lemma \ref{identity1}  can be rewritten as
\begin{equation}\label{stone}
{{s}^{3}}+{{q}^{3}}{{t}^{3}}=\frac{\psi {{({{q}^{3}})}^{4}}}{\psi ({{q}^{9}})},
\end{equation}
\begin{equation}\label{sttwo}
\frac{1}{\psi (q)}=\frac{\psi ({{q}^{9}})}{\psi {{({{q}^{3}})}^{4}}}({{s}^{2}}-qst+{{q}^{2}}{{t}^{2}}).
\end{equation}

To establish some congruences using $q$ series, we also need the following fact.
\begin{lemma}\label{basic}(Cf.\ \cite[Lemma 1.2]{radu}.)
Let $p$ be a prime and $\alpha $ be a positive integer. Then
\[(q;q)_{\infty }^{{{p}^{\alpha }}}\equiv ({{q}^{p}};{{q}^{p}})_{\infty }^{{{p}^{\alpha -1}}}  \pmod  {{{p}^{\alpha }}}.\]
\end{lemma}

Since $\psi (q)=\frac{({{q}^{2}};{{q}^{2}})_{\infty }^{2}}{{{(q;q)}_{\infty }}}$, we deduce from Lemma \ref{basic} that $\psi (q)^{p} \equiv \psi (q^{p})$ (mod $p$) for any prime $p$. We will use this fact many times.

We are going to establish an infinite family of congruences modulo 9.
\begin{theorem}\label{pod3mod9id}
We have
\begin{align} \nonumber
\begin{split}
\sum\limits_{n=0}^{\infty }{{{(-1)}^{n}}\mathrm{pod}_{-3}(3n){{q}^{n}}}&\equiv \frac{\psi ({{q}^{3}})}{\psi {{(q)}^{4}}} \pmod{9}.\\ \nonumber
\sum\limits_{n=0}^{\infty }{{{(-1)}^{n}}\mathrm{pod}_{-3}(3n+1){{q}^{n}}}&\equiv 3\psi (q)A{{(q)}^{2}} \pmod{9}. \\ \nonumber
\sum\limits_{n=0}^{\infty }{{{(-1)}^{n}}\mathrm{pod}_{-3}(3n+2){{q}^{n}}}&\equiv -3\psi ({{q}^{3}})\psi (q)A(q)  \pmod{9}.
\end{split}
\end{align}
\end{theorem}
\begin{proof}
By (\ref{sttwo}) we have
\begin{equation}\label{pod03exansion}
\sum\limits_{n=0}^{\infty }{{{(-1)}^{n}}\mathrm{pod}_{-3}(n){{q}^{n}}}=\frac{1}{\psi {{(q)}^{3}}}=\frac{\psi {{({{q}^{9}})}^{3}}}{\psi {{({{q}^{3}})}^{12}}}
{({{s}^{2}}-qst+{{q}^{2}}{{t}^{2}})}^{3}.
\end{equation}
Note that
\begin{equation}\label{sqt3}
{{({{s}^{2}}-qst+{{q}^{2}}{{t}^{2}})}^{3}}=({{s}^{6}}-7{{q}^{3}}{{s}^{3}}{{t}^{3}}+{{q}^{6}}{{t}^{6}})+q(-3{{s}^{5}}t+6{{q}^{3}}{{s}^{2}}{{t}^{4}})+{{q}^{2}}(6{{s}^{4}}{{t}^{2}}-3{{q}^{3}}s{{t}^{5}}).
\end{equation}
Applying (\ref{stone}) we obtain
\begin{displaymath}
\begin{split}
   \sum\limits_{n=0}^{\infty }{{{(-1)}^{3n}}\mathrm{pod}_{-3}(3n){{q}^{3n}}}
 & =\frac{\psi {{({{q}^{9}})}^{3}}}{\psi {{({{q}^{3}})}^{12}}}({{s}^{6}}-7{{q}^{3}}{{s}^{3}}{{t}^{3}}+{{q}^{6}}{{t}^{6}}) \\
 & \equiv \frac{\psi {{({{q}^{9}})}^{3}}}{\psi {{({{q}^{3}})}^{12}}}{{({{s}^{3}}+{{q}^{3}}{{t}^{3}})}^{2}} \\
 & \equiv \frac{\psi ({{q}^{9}})}{\psi {{({{q}^{3}})}^{4}}} \pmod{9}.
\end{split}
\end{displaymath}
Replacing ${{q}^{3}}$ by $q$,  we get the first congruence identity.

Similarly,
\begin{displaymath}
\begin{split}
   \sum\limits_{n=0}^{\infty }{{{(-1)}^{3n+1}}\mathrm{pod}_{-3}(3n+1){{q}^{3n+1}}}
 & =q\frac{\psi {{({{q}^{9}})}^{3}}}{\psi {{({{q}^{3}})}^{12}}}\big(-3{{s}^{2}}t({{s}^{3}}+{{q}^{3}}{{t}^{3}})+9{{q}^{3}}{{s}^{2}}{{t}^{4}}\big) \\
 & \equiv -3q\frac{\psi {{({{q}^{9}})}^{3}}}{\psi {{({{q}^{3}})}^{12}}}{{s}^{2}}t({{s}^{3}}+{{q}^{3}}{{t}^{3}}) \\
 & \equiv -3q\frac{\psi {{({{q}^{9}})}^{3}}}{\psi {{({{q}^{3}})}^{8}}} A{{({{q}^{3}})}^{2}} \pmod{9}.
\end{split}
\end{displaymath}
Dividing both sides by $-q$ , then replacing ${{q}^{3}}$ by $q$ and applying Lemma \ref{basic}, we obtain
\[\sum\limits_{n=0}^{\infty }{{{(-1)}^{n}}\mathrm{pod}_{-3}(3n+1){{q}^{n}}}\equiv 3\frac{\psi {{({{q}^{3}})}^{3}}}{\psi {{(q)}^{8}}} A{{(q)}^{2}}\equiv 3\psi (q)A{{(q)}^{2}} \pmod{9}. \]
This completes the proof of the second congruence identity.

In the same way,  we have
\begin{displaymath}
\begin{split}
  \sum\limits_{n=0}^{\infty }{{{(-1)}^{3n+2}}\mathrm{pod}_{-3}(3n+2){{q}^{3n+2}}}& ={{q}^{2}}\frac{\psi {{({{q}^{9}})}^{3}}}{\psi {{({{q}^{3}})}^{12}}}\big(9{{s}^{4}}{{t}^{2}}-3s{{t}^{2}}({{s}^{3}}+{{q}^{3}}{{t}^{3}})\big) \\
 & \equiv -3{{q}^{2}}\frac{\psi {{({{q}^{9}})}^{3}}}{\psi {{({{q}^{3}})}^{12}}}s{{t}^{2}}({{s}^{3}}+{{q}^{3}}{{t}^{3}}) \\
 & \equiv -3{{q}^{2}}\frac{\psi {{({{q}^{9}})}^{4}}}{\psi {{({{q}^{3}})}^{8}}}A({{q}^{3}}) \pmod{9}.
\end{split}
\end{displaymath}
Dividing both sides by ~${{q}^{2}}$,   then replacing  ${{q}^{3}}$ by $q$ and applying Lemma \ref{basic}, we obtain
\[\sum\limits_{n=0}^{\infty }{{{(-1)}^{n}}\mathrm{pod}_{-3}(3n+2){{q}^{n}}}\equiv -3\frac{\psi {{({{q}^{3}})}^{4}}}{\psi {{(q)}^{8}}}A(q)\equiv -3\psi ({{q}^{3}})\psi (q)A(q) \pmod{9}.\]
\quad
\end{proof}

\begin{theorem}\label{pod3mod9alpha}
For any integers $\alpha \ge 1$ and $n\ge 0$, we have
\begin{equation}\label{mod1}
\sum\limits_{n=0}^{\infty }{{{(-1)}^{n}}\mathrm{pod}_{-3}\Big({{3}^{2\alpha}}n+\frac{5\times {{3}^{2\alpha }}+3}{8}\Big){{q}^{n}}\equiv {{(-1)}^{\alpha -1}}\frac{\psi {{({{q}^{3}})}^{4}}}{\psi {{(q)}^{7}}}} \pmod{9}
\end{equation}
and
\begin{equation}\label{mod2}
\sum\limits_{n=0}^{\infty }{{{(-1)}^{n}}\mathrm{pod}_{-3}\Big({{3}^{2\alpha +1}}n+\frac{7\times {{3}^{2\alpha +1}}+3}{8}\Big){{q}^{n}}}\equiv {{(-1)}^{\alpha -1}}\frac{\psi {{({{q}^{3}})}^{5}}}{\psi {{(q)}^{8}}}  \pmod{9}.
\end{equation}
\end{theorem}
\begin{proof}
By (\ref{sttwo}) and  the first congruence in Theorem \ref{pod3mod9id},  we have
\begin{equation}\label{oneterm1}
\sum\limits_{n=0}^{\infty }{{{(-1)}^{n}}\mathrm{pod}_{-3}(3n){{q}^{n}}}\equiv \frac{\psi ({{q}^{3}})}{\psi {{(q)}^{4}}}=\frac{\psi {{({{q}^{9}})}^{4}}}{\psi {{({{q}^{3}})}^{15}}}{{({{s}^{2}}-qst+{{q}^{2}}{{t}^{2}})}^{4}} \pmod{9}.
\end{equation}
Note that
\begin{equation}\label{sqt4}
\begin{split}
{{({{s}^{2}}-qst+{{q}^{2}}{{t}^{2}})}^{4}}&=({{s}^{8}}-16{{q}^{3}}{{s}^{5}}{{t}^{3}}+10{{q}^{6}}{{s}^{2}}{{t}^{6}})+q(-4{{s}^{7}}t+19{{q}^{3}}{{s}^{4}}{{t}^{4}} -4{{q}^{6}}s{{t}^{7}})\\
&\quad +{{q}^{2}}(10{{s}^{6}}{{t}^{2}}-16{{q}^{3}}{{s}^{3}}{{t}^{5}}+{{q}^{6}}{{t}^{8}}).
\end{split}
\end{equation}
 Collecting all the terms of the form ${{q}^{3k+2}}$ in (\ref{oneterm1}),  we obtain
\begin{displaymath}
\begin{split}
 \sum\limits_{n=0}^{\infty }{{{(-1)}^{3n+2}}\mathrm{pod}_{-3}(9n+6){{q}^{3n+2}}} & =  {{q}^{2}}\frac{\psi {{({{q}^{9}})}^{4}}}{\psi {{({{q}^{3}})}^{15}}} (10{{s}^{6}}{{t}^{2}}-16{{q}^{3}}{{s}^{3}}{{t}^{5}}+{{q}^{6}}{{t}^{8}}) \\
 & \equiv {{q}^{2}}\frac{\psi {{({{q}^{9}})}^{4}}}{\psi {{({{q}^{3}})}^{15}}}{{t}^{2}}{{({{s}^{3}}+{{q}^{3}}{{t}^{3}})}^{2}} \\
  & \equiv {{q}^{2}}\frac{\psi {{({{q}^{9}})}^{4}}}{\psi {{({{q}^{3}})}^{7}}} \pmod{9} .
\end{split}
\end{displaymath}
Dividing both sides by ${{q}^{2}}$ and then replacing ${{q}^{3}}$ by $q$,  we get
exactly (\ref{mod1}) with $\alpha =1.$

Now suppose we already have (\ref{mod1}) for a given $\alpha$.
By (\ref{sttwo}) it follows that
\[\frac{\psi {{({{q}^{3}})}^{4}}}{\psi {{(q)}^{7}}}=\frac{\psi {{({{q}^{9}})}^{7}}}{\psi {{({{q}^{3}})}^{24}}}{{ ({{s}^{2}}-qst+{{q}^{2}}{{t}^{2}})}^{7}}.\]
Collecting all the terms of the form ${{q}^{3k+2}}$ in ${{({{s}^{2}}-qst+{{q}^{2}}{{t}^{2}})}^{7}}$,  we obtain
\begin{displaymath}
\begin{split}
  & {{q}^{2}}(28{{s}^{12}}t^{2}-266{{q}^{3}}{{s}^{9}}{{t}^{5}}+357{{q}^{6}}{{s}^{6}}{{t}^{8}}-77{{q}^{9}}{{s}^{3}}{{t}^{11}}+{{q}^{12}}{{t}^{14}}) \\
 & \equiv {{q}^{2}}{{t}^{2}}({{s}^{12}}+4{{q}^{3}}{{s}^{9}}{{t}^{3}}+6{{s}^{6}}{{q}^{6}}{{t}^{6}}+4{{q}^{9}}{{s}^{3}}{{t}^{9}}+{{q}^{12}}{{t}^{12}}) \\
 & \equiv {{q}^{2}}{{t}^{2}}{{({{s}^{3}}+{{q}^{3}}{{t}^{3}})}^{4}} \pmod{9}.
\end{split}
\end{displaymath}
Hence when collecting all the  terms of the form ${{q}^{3k+2}}$ in  (\ref{mod1}) and applying (\ref{stone}), we get
\[ \sum\limits_{n=0}^{\infty }(-1)^{3n+2}{\mathrm{pod}_{-3}\Big({{3}^{2\alpha }}(3n+2)+\frac{5\times {{3}^{2\alpha }}+3}{8}\Big){{q}^{3n+2}}} \equiv {{(-1)}^{\alpha -1}}{{q}^{2}}\frac{\psi {{({{q}^{9}})}^{5}}}{\psi {{({{q}^{3}})}^{8}}} \pmod{9}.\]

Dividing both sides by ${{q}^{2}}$ and then replacing ${{q}^{3}}$ by $q$, we deduce that
\begin{equation}\label{alpha2}
\sum\limits_{n=0}^{\infty }{{{(-1)}^{n}}\mathrm{pod}_{-3}\Big({{3}^{2\alpha +1}}n+\frac{7\times {{3}^{2\alpha +1}}+3}{8}\Big){{q}^{n}}}\equiv {{(-1)}^{\alpha -1}}\frac{\psi {{({{q}^{3}})}^{5}}}{\psi {{(q)}^{8}}} \pmod{9}.
\end{equation}

Again by (\ref{sttwo}) we have
\[\frac{\psi {{({{q}^{3}})}^{5}}}{\psi {{(q)}^{8}}}=\frac{\psi {{({{q}^{9}})}^{8}}}{\psi {{({{q}^{3}})}^{27}}}{{({{s}^{2}}-qst+{{q}^{2}}{{t}^{2}})}^{8}}.\]
Collecting all the terms  of the form ${{q}^{3k+1}}$ in ${{({{s}^{2}}-qst+{{q}^{2}}{{t}^{2}})}^{8}}$, we get
\begin{displaymath}
\begin{split}
  & \quad q(-8{{s}^{15}}t+266{{q}^{3}}{{s}^{12}}{{t}^{4}}-1016{{q}^{6}}{{s}^{9}}{{t}^{7}}+784{{q}^{9}}{{s}^{6}}{{t}^{10}}-112{{q}^{12}}{{s}^{3}}{{t}^{13}}+{{q}^{15}}{{t}^{16}}) \\
 & \equiv qt({{s}^{15}}+5{{q}^{3}}{{s}^{12}}{{t}^{3}}+10{{q}^{6}}{{s}^{9}}{{t}^{6}}+10{{q}^{9}}{{s}^{6}}{{t}^{9}}+5{{q}^{12}}{{s}^{3}}{{t}^{12}}+{{q}^{15}}{{t}^{15}}) \\
 & \equiv qt{{({{s}^{3}}+{{q}^{3}}{{t}^{3}})}^{5}}  \pmod{9}.
\end{split}
\end{displaymath}
If we collect all the terms of the form ${{q}^{3k+1}}$ in (\ref{alpha2}), we obtain
\[ \sum\limits_{n=0}^{\infty }{{{(-1)}^{3n+1}}\mathrm{pod}_{-3}\Big({{3}^{2\alpha +1}}(3n+1)+\frac{7\times {{3}^{2\alpha +1}}+3}{8}\Big){q}^{3n+1}} \equiv {{(-1)}^{\alpha -1}}q\frac{\psi {{({{q}^{9}})}^{4}}}{\psi {{({{q}^{3}})}^{7}}} \pmod{9}. \]
Now dividing both sides by $-q$ and then replacing ${{q}^{3}}$ by $q$,  we get
\[\sum\limits_{n=0}^{\infty }{{{(-1)}^{n}}\mathrm{pod}_{-3}\Big({{3}^{2\alpha +2}}n+\frac{5\times {{3}^{2\alpha +2}}+3}{8}\Big){{q}^{n}}}
 ={{(-1)}^{(\alpha +1)-1}}\frac{\psi {{({{q}^{3}})}^{4}}}{\psi {{(q)}^{7}}} \pmod{9}.\]
By doing  induction on $\alpha $, we complete our proof.
\end{proof}
\begin{corollary}
For any integer $n\ge 0$, we have
\begin{displaymath}
\begin{split}
  & \mathrm{pod}_{-3}(81n+51)\equiv -\mathrm{pod}_{-3}(9n+6) \pmod{9}, \\
 & \mathrm{pod}_{-3}(27n+24)\equiv -\mathrm{pod}_{-3}(243n+213) \pmod{9}.
\end{split}
\end{displaymath}
\end{corollary}
\begin{proof}
Let $\alpha =1,2$ in (\ref{mod1}). We obtain
\[\sum\limits_{n=0}^{\infty }{{{(-1)}^{n}}\mathrm{pod}_{-3}(9n+6){{q}^{n}}} \equiv \frac{\psi {{({{q}^{3}})}^{4}}}{\psi {{(q)}^{7}}} \equiv - \sum\limits_{n=0}^{\infty }{{{(-1)}^{n}}\mathrm{pod}_{-3}(81n+51){{q}^{n}}} \pmod{9}.\]
The first congruence follows immediately.

The second congruence can be obtained in a similar way by letting $\alpha =1, 2$ in (\ref{mod2}).
\end{proof}

\begin{lemma}\label{psi8}
The coefficients of ${{q}^{3n+2}}$ in the series expansion of $\frac{1}{\psi {{(q)}^{8}}}$ is divisible by $9$.
\end{lemma}
\begin{proof}
By  (\ref{sttwo}) we have
\[\frac{1}{\psi {{(q)}^{8}}}=\frac{\psi {{({{q}^{9}})}^{8}}}{\psi {{({{q}^{3}})}^{32}}}{\big(s^{2}-qst+q^2t^2\big)}^{8}.\]
Collecting all the terms of the form ${{q}^{3k+2}}$ in ${{({{s}^{2}}-qst+{{q}^{2}}{{t}^{2}})}^{8}}$,  we obtain
\[9{{q}^{2}}(4{{s}^{14}}{{t}^{2}}-56{{q}^{3}}{{s}^{11}}{{t}^{5}}+123{{q}^{6}}{{s}^{8}}{{t}^{8}}-56{{q}^{9}}{{s}^{5}}{{t}^{11}}+4{{q}^{12}}{{s}^{2}}{{t}^{14}}).\]
It is obvious that all the coefficients are divisible by 9.
\end{proof}
Now we are able to prove the following theorem.
\begin{theorem}\label{pod3mod9}
For any integers $\alpha \ge 1$ and $n\ge 0$, we have
\[\mathrm{pod}_{-3}\Big({{3}^{2\alpha +2}}n+\frac{23\times {{3}^{2\alpha +1}}+3}{8}\Big)\equiv 0 \pmod{9}.\]
\end{theorem}
\begin{proof}
By Lemma \ref{psi8}, we know the coefficient of the term ${{q}^{3n+2}}$ in the left hand side of (\ref{mod2}) is divisible by 9. Hence we have
\[\mathrm{pod}_{-3}\Big({{3}^{2\alpha +1}}(3n+2)+\frac{7\times {{3}^{2\alpha +1}}+3}{8}\Big)\equiv 0 \pmod{9}.\]
This is the desired result we want.
\end{proof}
Note that for different integers $\alpha  \ge 1$,  the arithmetic sequences $\big\{{{3}^{2\alpha +2}}n+\frac{23\times {{3}^{2\alpha +1}}+3}{8}: n=0,1,2, \cdots \big\}$ mentioned in Theorem \ref{pod3mod9} are disjoint, and they account for
\[\frac{1}{{{3}^{4}}}+\frac{1}{{{3}^{6}}}+\cdots =\frac{1}{72}\]
of all nonnegative integers. Thus we have the following corollary.
\begin{corollary}
 $\mathrm{pod}_{-3}(n)$  is divisible by $9$ for  at least $1/72$ of all nonnegative integers.
\end{corollary}

\section{Arithmetic Relations Between $\mathrm{pod}(n)$ and $\mathrm{pod}_{-3}(n)$}

There are  some nontrivial arithmetic relations between $\mathrm{pod}(n)$ and $\mathrm{pod}_{-3}(n)$. First we recall some congruence identities discovered  in \cite{hisc}.
\begin{lemma}\label{pod3n2}(Cf.\ \cite[Section 3]{hisc}.)
$\sum\limits_{n=0}^{\infty }{{{(-1)}^{n}}\mathrm{pod}(3n+2){{q}^{n}}}=\frac{\psi {{({{q}^{3}})}^{3}}}{\psi {{(q)}^{4}}}.$
\end{lemma}
\begin{lemma}\label{podmod9id}(Cf.\ \cite[Section 4]{hisc}.)
\begin{displaymath}
\begin{split}
\sum\limits_{n=0}^{\infty }{{{(-1)}^{n}}\mathrm{pod}(9n+2){{q}^{n}}}&\equiv A{{(q)}^{2}}\frac{\psi {{({{q}^{3}})}^{2}}}{\psi {{(q)}^{5}}} \pmod{9},\\
\sum\limits_{n=0}^{\infty }{{{(-1)}^{n}}\mathrm{pod}(9n+5){{q}^{n}}}&\equiv 4A(q)\frac{\psi {{({{q}^{3}})}^{3}}}{\psi {{(q)}^{5}}} \pmod{9},\\
\sum\limits_{n=0}^{\infty }{{{(-1)}^{n}}\mathrm{pod}(9n+8){{q}^{n}}}&\equiv \frac{\psi {{({{q}^{3}})}^{4}}}{\psi {{(q)}^{5}}} \pmod{9}.
\end{split}
\end{displaymath}
\end{lemma}

\begin{theorem}
For any integer $n\ge 0$,  we have
\begin{align} \nonumber
  & \mathrm{pod}_{-3}(3n+2)\equiv 6 \, \mathrm{pod}(9n+5)  \pmod{9}, \\ \nonumber
 & \mathrm{pod}_{-3}(3n+1)\equiv 3 \, \mathrm{pod}(9n+2) \pmod{18}.
\end{align}
\end{theorem}
\begin{proof}
By Lemma \ref{basic}, we have $\psi {{({{q}^{3}})}^{2}}\equiv \psi {{(q)}^{6}}$ (mod $3$). Hence
\[24A(q)\frac{\psi {{({{q}^{3}})}^{3}}}{\psi {{(q)}^{5}}}\equiv -3\psi ({{q}^{3}})\psi (q)A(q) \pmod{9}.\]
Hence by the second identity in Lemma \ref{podmod9id} and the third identity in Theorem \ref{pod3mod9id},  we get
\[\mathrm{pod}_{-3}(3n+2)\equiv 6 \, \mathrm{pod}(9n+5) \pmod{9}.\]

Now we turn to the proof of the second congruence.  Since $\psi {{({{q}^{3}})}^{2}}\equiv \psi {{(q)}^{6}}$ (mod $3$), we have
\[3A{{(q)}^{2}}\frac{\psi {{({{q}^{3}})}^{2}}}{\psi {{(q)}^{5}}}\equiv 3A{{(q)}^{2}}\psi (q)  \pmod{9}.\]
Hence by the first identity in Lemma \ref{podmod9id} and the second identity in Theorem \ref{pod3mod9id}, we obtain
\[\mathrm{pod}_{-3}(3n+1)\equiv 3 \, \mathrm{pod}(9n+2) \pmod{9}.\]

  The remaining task is to show that $\mathrm{pod}_{-3}(3n+1)\equiv \mathrm{pod}(9n+2)$ (mod $2$).

  By Lemma \ref{pod3n2} and (\ref{sttwo}),  we have
\begin{equation}\label{twoterm1}
\sum\limits_{n=0}^{\infty }{{{(-1)}^{n}}\mathrm{pod}(3n+2){{q}^{n}}}=\frac{\psi {{({{q}^{3}})}^{3}}}{\psi {{(q)}^{4}}}=\frac{\psi {{({{q}^{9}})}^{4}}}{\psi {{({{q}^{3}})}^{13}}}{{({{s}^{2}}-qst+{{q}^{2}}{{t}^{2}})}^{4}}.
\end{equation}
Collecting all the  terms  of the form ${{q}^{3n}}$ in the right hand side of  (\ref{twoterm1}),  by (\ref{sqt4}) we obtain
\[\frac{\psi {{({{q}^{9}})}^{4}}}{\psi {{({{q}^{3}})}^{13}}}({{s}^{8}}-16{{q}^{3}}{{s}^{5}}{{t}^{3}}+10{{q}^{6}}{{s}^{2}}{{t}^{6}})\equiv \frac{\psi {{({{q}^{9}})}^{4}}}{\psi {{({{q}^{3}})}^{13}}}A{{({{q}^{3}})}^{8}} \pmod{2}.\]
Hence we have
\[\sum\limits_{n=0}^{\infty }{{{(-1)}^{3n}}\mathrm{pod}(9n+2){{q}^{3n}}}\equiv \frac{\psi {{({{q}^{9}})}^{4}}}{\psi {{({{q}^{3}})}^{13}}}A{{({{q}^{3}})}^{8}} \pmod{2}.\]
Replacing ${{q}^{3}}$ by $q$,  we get
\begin{equation}\label{podmod2}
\sum\limits_{n=0}^{\infty }{{{(-1)}^{n}}\mathrm{pod}(9n+2){{q}^{n}}}\equiv \frac{\psi {{({{q}^{3}})}^{4}}}{\psi {{(q)}^{13}}}A{{(q)}^{8}} \pmod{2}.
\end{equation}

Back to the expression of $\sum\limits_{n=0}^{\infty }{{{(-1)}^{n}}\mathrm{pod}_{-3}(n){{q}^{n}}}$ in (\ref{pod03exansion}),  from (\ref{sqt3}) we have
\begin{align} \nonumber
   \sum\limits_{n=0}^{\infty }{{{(-1)}^{3n+1}}\mathrm{pod}_{-3}(3n+1){{q}^{3n+1}}}
 & =q\frac{\psi {{({{q}^{9}})}^{3}}}{\psi {{({{q}^{3}})}^{12}}}(-3{{s}^{5}}t+6{{q}^{3}}{{s}^{2}}{{t}^{4}})\\ \nonumber
 & \equiv q\frac{\psi {{({{q}^{9}})}^{4}}}{\psi {{({{q}^{3}})}^{12}}}A{{({{q}^{3}})}^{5}} \pmod{2}.
\end{align}
Now dividing both sides by $-q$ and replacing ${{q}^{3}}$ by $q$, we get
\begin{equation}\label{pod-3mod2}
\sum\limits_{n=0}^{\infty }{{{(-1)}^{n}}\mathrm{pod}_{-3}(3n+1){{q}^{n}}}\equiv \frac{\psi {{({{q}^{3}})}^{4}}}{\psi {{(q)}^{12}}}A{{(q)}^{5}} \pmod{2}.
\end{equation}
By Lemma \ref{basic}, we have ${{({{q}^{2}};{{q}^{2}})}_{\infty }}\equiv {{(q;q)}_{\infty}^{2}}$ (mod $2$) and $({{q}^{3}};{{q}^{3}})_{\infty }^{2}\equiv {{({{q}^{6}};{{q}^{6}})}_{\infty }}$ (mod $2$). Hence
\[A(q)=\frac{{{({{q}^{2}};{{q}^{2}})}_{\infty }}({{q}^{3}};{{q}^{3}})_{\infty }^{2}}{{{(q;q)}_{\infty }}{{({{q}^{6}};{{q}^{6}})}_{\infty }}}\equiv {{(q;q)}_{\infty }} \pmod{2},\]
and
\[\psi (q)=\frac{({{q}^{2}};{{q}^{2}})_{\infty }^{2}}{{{(q;q)}_{\infty }}}\equiv (q;q)_{\infty }^{3} \pmod{2}.\]
Therefore, we have $A{{(q)}^{3}}\equiv \psi (q)$ (mod $2$). Thus
\[\frac{\psi {{({{q}^{3}})}^{4}}}{\psi {{(q)}^{13}}}A{{(q)}^{8}}\equiv \frac{\psi {{({{q}^{3}})}^{4}}}{\psi {{(q)}^{12}}}A{{(q)}^{5}} \pmod{2}.\]
From (\ref{podmod2}), (\ref{pod-3mod2}) and the congruence above, we deduce that $\mathrm{pod}_{-3}(3n+1)\equiv \mathrm{pod}(9n+2)$ (mod $2$), which completes our proof.
\end{proof}

\section{Congruences Modulo 7 And 11}

For nonnegative integers $n$ and $k$,  let ${{r}_{k}}(n)$ denote the number of representations of $n$ as sum of $k$ squares, and ${{t}_{k}}(n)$ denote  the number of representations of $n$ as  sum of $k$ triangular numbers. Let $\sigma (n)$ denote the sum of positive divisors of $n$.

The method used in this section is quite similar to the one used in \cite{Wang2}, where the author applied the properties of $r_{k}(n)$ to give new congruences satisfied by overpartition function.
\begin{lemma}\label{t4t8}
For any prime $p \ge 3$,  we have
\[{{t}_{4}}\Big(pn+\frac{p-1}{2}\Big)\equiv {{t}_{4}}(n) \pmod{p},  \quad  {{t}_{8}}(pn+p-1)\equiv {{t}_{8}}(n) \pmod{p}.\]
\end{lemma}
\begin{proof}
From Theorem  3.6.3 in \cite{Bruce}, we know ${{t}_{4}}(n)=\sigma (2n+1).$
For any positive integer $N$, we have
\[\sigma (N)=\sum\limits_{d|N, \, p|d}{d}+\sum\limits_{d|N, \,p \nmid d}{d}\equiv \sum\limits_{d|N, \, p \nmid d}{d} \pmod{p}.\]
Let $N=2n+1$ and $N=p(2n+1)$ respectively.  It is  easy to see that $\sigma (p(2n+1))\equiv \sigma (2n+1)$ (mod $p$), which clearly implies the first congruence.

 From equation (3.8.3) in page 81 of \cite{Bruce}, we know
 \[{{t}_{8}}(n)=\sum\limits_{\begin{smallmatrix}
 d|(n+1) \\
 d \, \mathrm{odd}
\end{smallmatrix}}{{{\Big(\frac{n+1}{d}\Big)}^{3}}}. \]
Hence similarly we can prove the second congruence.
\end{proof}
\begin{lemma}\label{rsrelate}(Cf.\ \cite{relation}.)
For $1\le k\le 7$, we have
\[{{r}_{k}}(8n+k)={{2}^{k}}\Big(1+\frac{1}{2}{k \choose 4}\Big){{t}_{k}}(n) .\]
\end{lemma}

\begin{theorem}\label{pod3mod7}
For any integer $n\ge 0$, we have
\[\mathrm{pod}_{-3}(7n+3)\equiv {{(-1)}^{n+1}}{{r}_{3}}(8n+3) \pmod{7}.\]
\end{theorem}
\begin{proof}
Let $p=7$ in Lemma \ref{t4t8}. We deduce that ${{t}_{4}}(7n+3)\equiv {{t}_{4}}(n)$ (mod $7$).

From definition we have
\[\psi {{(q)}^{7}}\sum\limits_{n=0}^{\infty }{\mathrm{pod}_{-3}(n){{(-q)}^{n}}}=\psi {{(q)}^{4}}=\sum\limits_{n=0}^{\infty }{{{t}_{4}}(n){{q}^{n}}}.\]
By Lemma \ref{basic}, we have $\psi {{(q)}^{7}}\equiv \psi ({{q}^{7}})$ (mod $7$). Replacing $\psi {{(q)}^{7}}$ by $\psi ({{q}^{7}})$ modulo 7 and extracting those terms of the form ${{q}^{7n+3}}$ in the above identity, we obtain
\[\psi ({{q}^{7}})\sum\limits_{n=0}^{\infty }{\mathrm{pod}_{-3}(7n+3){{(-q)}^{7n+3}}}\equiv \sum\limits_{n=0}^{\infty }{{{t}_{4}}(7n+3){{q}^{7n+3}}} \pmod{7}.\]
  Dividing both sides by $-{{q}^{3}}$, then replacing ${{q}^{7}}$ by $q$ and replacing $t_{4}(7n+3)$ by $t_{4}(n)$,  we get
\[\sum\limits_{n=0}^{\infty }{\mathrm{pod}_{-3}(7n+3){{(-q)}^{n}}}\equiv -\psi {{(q)}^{3}}\equiv -\sum\limits_{n=0}^{\infty }{{{t}_{3}}(n){{q}^{n}}} \pmod{7}.\]
Hence we have $\mathrm{pod}_{-3}(7n+3)\equiv {{(-1)}^{n+1}}{{t}_{3}}(n)$ (mod $7$).

Let $k=3$ in Lemma \ref{rsrelate}. We obtain ${{t}_{3}}(n)={{r}_{3}}(8n+3)/8$,  from which the theorem follows.
\end{proof}

We can deduce some interesting congruences from this theorem.

\begin{lemma}\label{r3relation}(Cf.\ \cite{3square}.)
Let $p \ge 3$ be a prime. For any  integers $n \ge 1$  and $\alpha \ge 0$,  we have
\[{{r}_{3}}({{p}^{2 \alpha }}n)=\Bigg(\frac{{{p}^{\alpha +1}}-1}{p-1}-\Big(\frac{-n}{p}\Big)\frac{{{p}^{\alpha }}-1}{p-1}\Bigg){{r}_{3}}(n)-p\frac{{{p}^{\alpha }}-1}{p-1}{{r}_{3}}(n/{{p}^{2}}).\]
Here $(\frac{\cdot }{p})$ denotes the Legendre symbol,  and we take ${{r}_{3}}(n/{{p}^{2}})=0$ unless  ${{p}^{2}} | n$.
\end{lemma}

\begin{theorem}
Let $p$ be a prime and $p\equiv 6$ \text{\rm{(mod $7$)}}. Let $N$ be a positive integer which is coprime to $p$ and ${{p}^{3}}N\equiv 3$ \text{\rm{(mod $8$)}}. We have
\[\mathrm{pod}_{-3}\Big(\frac{7{{p}^{3}}N+3}{8}\Big)\equiv 0 \pmod{7}.\]
\end{theorem}

 \begin{proof}
Let $\alpha =1$ and $n=pN$ in Lemma \ref{r3relation}. We have
\[{{r}_{3}}({{p}^{3}}N)=(p+1){{r}_{3}}(pN)\equiv 0  \pmod{7}.\]
 Let $n=\frac{{{p}^{3}}N-3}{8}$.   By Theorem \ref{pod3mod7}, we deduce that
\[\mathrm{pod}_{-3}\Big(\frac{7{{p}^{3}}N+3}{8}\Big)=\mathrm{pod}_{-3}(7n+3)\equiv {{(-1)}^{n+1}}{{r}_{3}}({{p}^{3}}N)\equiv 0  \pmod{7}.\]
\end{proof}
For example,  let $p=13$ and $N=8n+7$,  where $n \not\equiv 4$ (mod $13$) is a nonnegative integer. We establish a Ramanujan-type congruence as follows:
\[\mathrm{pod}_{-3}(15379n+13457)\equiv 0  \pmod{7}.\]

\begin{theorem}
Let $p \ge 3$ be a prime,  $N$ a positive integer which is coprime to $p$ and $pN\equiv 3$ \text{\rm{(mod $8$)}}. Let $\alpha $ be any nonnegative integer.\\
(1) If $p\equiv 1$  \text{\rm{(mod $7$)}}, then
\[\mathrm{pod}_{-3}\Big(\frac{7{{p}^{14\alpha +13}}N+3}{8}\Big)\equiv 0 \pmod{7}.\]
(2) If $p \not\equiv 0$ \text{\rm{or}} $1$ \text{\rm{(mod $7$)}}, then
\[\mathrm{pod}_{-3}\Big(\frac{7{{p}^{12\alpha +11}}N+3}{8}\Big)\equiv 0 \pmod{7}.\]
\end{theorem}
\begin{proof}
(1) Let $n=pN$ in Lemma \ref{r3relation},  and then we replace $\alpha $ by $7\alpha +6$.  Since $p\equiv 1$ (mod $7$),   we have
\[\frac{{{p}^{7\alpha +7}}-1}{p-1}=1+p+\cdots +{{p}^{7\alpha +6}}\equiv 0 \pmod{7}.\]
Hence ${{r}_{3}}({{p}^{14\alpha +13}}N)\equiv 0$ (mod $7$).

Let $n=\frac{{{p}^{14\alpha +13}}N-3}{8}$.  By Theorem \ref{pod3mod7}, we deduce that
\[\mathrm{pod}_{-3}\Big(\frac{7{{p}^{14\alpha +13}}N+3}{8}\Big)=\mathrm{pod}_{-3}(7n+3)\equiv {{(-1)}^{n+1}}{{r}_{3}}({{p}^{14\alpha +13}}N)\equiv 0 \pmod{7}.\]

(2) Let $n=pN$ in Lemma \ref{r3relation}, and then we replace $\alpha $ by $6\alpha +5$.  Note that ${{p}^{6\alpha +6}}\equiv 1$ (mod $7$),  we get ${{r}_{3}}({{p}^{12\alpha +11}}N)\equiv 0$ (mod $7$).

Let $n=\frac{{{p}^{12\alpha +11}}N-3}{8}$.  From Theorem \ref{pod3mod7} we obtain
\begin{displaymath}
\mathrm{pod}_{-3}\Big(\frac{7{{p}^{12\alpha +11}}N+3}{8}\Big)=\mathrm{pod}_{-3}(7n+3)\equiv {{(-1)}^{n+1}}{{r}_{3}}({{p}^{12\alpha +11}}N)\equiv 0 \pmod{7}.
\end{displaymath}
\quad
\end{proof}

\begin{theorem}
Let $a\in \{3,19,27\}$.  For any integers $n\ge 0$ and $\alpha \ge 1$,  we have
\[\mathrm{pod}_{-3}\Big({{7}^{2\alpha +2}}n+\frac{a\cdot {{7}^{2\alpha +1}}+3}{8}\Big)\equiv 0 \pmod{7}.\]
\end{theorem}
\begin{proof}
Let $p=7$ and $n =7m+r$ ($r\in \{3,5,6\}$) in Lemma \ref{r3relation}.  It is easy to deduce that ${{r}_{3}}({{7}^{2\alpha }}(7m+r))\equiv 0$ (mod $7$).

Let $n={{7}^{2\alpha +1}}m+\frac{a\cdot {{7}^{2\alpha }}-3}{8}$ ($a\in \{3,19,27\}$) in Theorem \ref{pod3mod7}. We have
\[\mathrm{pod}_{-3}\Big({{7}^{2\alpha +2}}m+\frac{a\cdot {{7}^{2\alpha +1}}+3}{8}\Big)=\mathrm{pod}_{-3}(7n+3)\equiv {{{(-1)}^{n+1}}}{{r}_{3}}(8n+3) \pmod{7}.\]
Since ${{r}_{3}}(8n+3)={{r}_{3}}({{7}^{2\alpha }}(56m+a))\equiv 0$ (mod $7$), we complete our proof.
\end{proof}

Note that for different pairs $(\alpha ,a)$, $\alpha \ge 1$ and $ a\in \{3,19,27\}$, the arithmetic sequences
$\big\{{{7}^{2\alpha +2}}n+\frac{a\cdot {{7}^{2\alpha +1}}+3}{8}:n=0,1,2,\cdots \big\}$ are disjoint, and they account for
\[3\Big(\frac{1}{{{7}^{4}}}+\frac{1}{{{7}^{6}}}+\cdots +\frac{1}{7^{2\alpha+2}}+ \cdots \Big)=\frac{1}{784}\]
of all nonnegative integers. Thus we have
\begin{corollary}\label{least}
$\mathrm{pod}_{-3}(n)$ is divisible by $7$ for at least $1/784$ of all nonnegative integers.
\end{corollary}

\begin{theorem}\label{pod3mod11}
For any integers $n\ge 0$,  we have
\[\mathrm{pod}_{-3}(11n+10)\equiv {{4(-1)}^{n}}{{r}_{7}}(8n+7)  \pmod{11}.\]
\end{theorem}
\begin{proof}
Let $p=11$ in Lemma \ref{t4t8}. We deduce that ${{t}_{8}}(11n+10)\equiv {{t}_{8}}(n)$ (mod $11$). By Lemma \ref{basic} we have $\psi {{(q)}^{11}} \equiv \psi ({{q}^{11}})$ (mod $11$).  Hence
\[\psi ({{q}^{11}})\sum\limits_{n=0}^{\infty }{\mathrm{pod}_{-3}(n){{(-q)}^{n}}}\equiv \psi {{(q)}^{8}}=\sum\limits_{n=0}^{\infty }{{{t}_{8}}(n){{q}^{n}}} \, \pmod{11}.\]
If we extract all the terms of the form ${{q}^{11n+10}}$, we obtain
\[\psi ({{q}^{11}})\sum\limits_{n=0}^{\infty }{\mathrm{pod}_{-3}(11n+10){{(-q)}^{11n+10}}}\equiv \sum\limits_{n=0}^{\infty }{{{t}_{8}}(11n+10){{q}^{11n+10}}} \pmod{11}.\]
Dividing both sides by ${{q}^{10}}$ and replacing ${{q}^{11}}$ by  $q$,  we have
\[\psi (q)\sum\limits_{n=0}^{\infty }{\mathrm{pod}_{-3}(11n+10){{(-q)}^{n}}}\equiv \sum\limits_{n=0}^{\infty }{{{t}_{8}}(11n+10){{q}^{n}}}\equiv \sum\limits_{n=0}^{\infty }{{{t}_{8}}(n){{q}^{n}}}=\psi {{(q)}^{8}} \pmod{11}.\]
Therefore,
\[\sum\limits_{n=0}^{\infty }{\mathrm{pod}_{-3}(11n+10){{(-q)}^{n}}}\equiv \psi {{(q)}^{7}}=\sum\limits_{n=0}^{\infty }{{{t}_{7}}(n){{q}^{n}}} \pmod{11}.\]
Comparing the coefficients of ${{q}^{n}}$ on both sides, we obtain
\[\mathrm{pod}_{-3}(11n+10)\equiv {{(-1)}^{n}}{{t}_{7}}(n) \pmod{11}.\]
Let $k=7$ in  Lemma \ref{rsrelate}. We deduce that ${{t}_{7}}(n)=\frac{1}{2368}{{r}_{7}}(8n+7)\equiv 4{{r}_{7}}(8n+7)$ (mod $11$). The theorem follows.
\end{proof}

\begin{lemma}\label{r7re}(Cf.\ \cite{Cooper}.)
Let $p \ge 3$ be a prime,  and $n$  a nonnegative integer such that  ${{p}^{2}} \nmid n$. For any integer $\alpha \ge 0$,  we have
\[{{r}_{7}}({{p}^{2\alpha }}n)=\Bigg(\frac{{{p}^{5\alpha +5}}-1}{{{p}^{5}}-1}-{{p}^{2}}\big(\frac{-n}{p}\big)\frac{{{p}^{5\alpha }}-1}{{{p}^{5}}-1}\Bigg){{r}_{7}}(n).\]
\end{lemma}

\begin{theorem}\label{twomod11}
Let $p \ge 3$ be a prime  and  $\alpha \ge 0$ be an integer.  Suppose $1\le r\le 7$ and $pr\equiv 7$ \text{\rm{(mod $8$)}}, $0\le a<p$ and $p \nmid 8a+r$, where both $a$ and $r$ are integers. \\
(1)  If $p\equiv 1,3,4,5,9$ \text{\rm{(mod $11$)}},  then
\[\mathrm{pod}_{-3}\Big(11\cdot {{p}^{22\alpha +22}}n+\frac{(88a+11r){{p}^{22\alpha +21}}+3}{8}\Big)\equiv 0 \pmod{11}.\]
(2)  If $p\equiv 2,6,7,8,10$ \text{\rm{(mod $11$)}},   then
\[\mathrm{pod}_{-3}\Big(11\cdot {{p}^{4\alpha +4}}n+\frac{(88a+11r){{p}^{4\alpha +3}}+3}{8}\Big)\equiv 0 \pmod{11}.\]
\end{theorem}
\begin{proof}
(1) Replacing $\alpha $ by $11\alpha +10$ and setting $n=pN$ in Lemma  \ref{r7re},  where $N$ is coprime to $p$. Note that $p\equiv 1,3,4,5,9$ (mod $11$) implies ${{p}^{5}}\equiv 1$ (mod $11$),  we have
\[\frac{{{p}^{55\alpha +55}}-1}{{{p}^{5}}-1}=1+{{p}^{5}}+{{p}^{10}}+\cdots +{{p}^{5(11\alpha +10)}}\equiv 0 \pmod{11}.\]
From which we know ${{r}_{7}}({{p}^{22\alpha +21}}N)\equiv 0$ (mod $11$).

Let $n={{p}^{22\alpha +22}}m+\frac{(8a+r){{p}^{22\alpha +21}}-7}{8}$.  We have
\[{{r}_{7}}(8n+7)={{r}_{7}}\big({{p}^{22\alpha +21}}(8pm+8a+r)\big)\equiv 0 \pmod{11}.\]
 By Theorem \ref{pod3mod11}, we obtain
 \[\mathrm{pod}_{-3}\Big(11\cdot {{p}^{22\alpha +22}}m+\frac{(88a+11r){{p}^{22\alpha +21}}+3}{8}\Big)=\mathrm{pod}_{-3}(11n+10)\equiv 0 \pmod{11}.\]

(2)  Let $n=pN$,  where $N$ is coprime to $p$, and we replace $\alpha $ by $2\alpha +1$ in Lemma \ref{r7re}.  Note that $p\equiv 2,6,7,8,10$ (mod  $11$) implies ${{p}^{5}}\equiv -1$ (mod $11$), we have ${{p}^{10\alpha +10}}\equiv 1$ (mod $11$).    We deduce that ${{r}_{7}}({{p}^{4\alpha +3}}N)\equiv 0$ (mod  $11$).

Let $n={{p}^{4\alpha +4}}m+\frac{(8a+r){{p}^{4\alpha +3}}-7}{8}$.  We have
\[{{r}_{7}}(8n+7)={{r}_{7}}\big({{p}^{4\alpha +3}}(8pm+8a+r)\big)\equiv 0 \pmod{11}.\]
By Theorem \ref{pod3mod11}, we obtain
 \[\mathrm{pod}_{-3}\Big(11\cdot {{p}^{4\alpha +4}}m+\frac{(88a+11r){{p}^{4\alpha +3}}+3}{8}\Big)=\mathrm{pod}_{-3}(11n+10)\equiv 0 \pmod{11}.\]
\quad
\end{proof}
For example,  let  $p=7$ in part (2) of Theorem \ref{twomod11}.  Accordingly we set $r=1$ and $a\in \{0,1,2,3,4,5\}$.  For any integer $n\ge 0$, we have
\[\mathrm{pod}_{-3}\Big(11\cdot {{7}^{4\alpha +4}}n+\frac{(88a+11)\cdot {{7}^{4\alpha +3}}+3}{8}\Big)\equiv 0 \pmod{11}.\]
Starting from the above congruence,  we can prove the following corollary by a  way similar to the proof of Corollary \ref{least},
\begin{corollary}
$\mathrm{pod}_{-3}(n)$ is divisible by $11$ for at least $1/4400$ of all nonnegative integers.
\end{corollary}

\end{document}